\documentclass[12pt]{amsart}
\usepackage[headings]{fullpage}

\usepackage{amsmath,amssymb,amsthm}
\usepackage{graphicx}
\usepackage{enumerate}
\usepackage{url}
\usepackage{slashed}
\usepackage{bm}
\usepackage{array}

\usepackage[german,english]{babel}

\usepackage[bookmarks=true,%
    colorlinks=true,%
    linkcolor=blue,%
    citecolor=blue,%
    filecolor=blue,%
    menucolor=blue,%
    urlcolor=blue,%
    breaklinks=true]{hyperref}

\usepackage[all]{xy}
\usepackage{verbatim}

\newtheorem{thm}{Theorem}[]
\newtheorem*{thm*}{Theorem}

\newtheorem{ex}[thm]{Example}

\newcommand{\param}{{\mathchoice{\mkern1mu\mbox{\raise2.2pt\hbox{$
\centerdot$}}
\mkern1mu}{\mkern1mu\mbox{\raise2.2pt\hbox{$\centerdot$}}\mkern1mu}{
\mkern1.5mu\centerdot\mkern1.5mu}{\mkern1.5mu\centerdot\mkern1.5mu}}}

\begin{document}
\title{Existence and deformability of topological Morse functions}
\author{Ingrid Irmer}
\address{SUSTech International Center for Mathematics\\
Southern University of Science and Technology\\Shenzhen, China
}
\address{Department of Mathematics\\
Southern University of Science and Technology\\Shenzhen, China
}
\email{ingridmary@sustech.edu.cn}

\today

\begin{abstract}
In the 1950s Morse defined the analogue of Morse functions for topological manifolds. In many instances, when mathematicians are using techniques on topological manifolds that appear to be Morse-theoretic in nature, there is a topological Morse function implicit in the argument. Topological Morse functions are known to inherit most of the familiar properties of the usual (smooth) Morse functions, with two crucial exceptions: existence and deformability. This paper gives a simple construction of continuous families of topological Morse functions.
\end{abstract}

\maketitle


\section{Introduction and history of the problem}
\label{intro}
In \cite{Morse}, Morse functions were defined for topological manifolds. A definition is given in Section \ref{secdefns}. When a topological Morse function exists, it can be used in most of the same ways as a smooth Morse function for studying the topology of a manifold. Unfortunately, constructing topological Morse functions or even proving their existence is more difficult than in the case of smooth functions on compact manifolds, as topological Morse functions are not known to be generic, and it is an open question from \cite{Morse} whether there exists a topological Morse function on every topological manifold. The contribution of this paper is to give a construction of continuous families of topological Morse functions. All the examples in the literature of which the author is aware take this form, and it seems likely that all topological Morse functions can be obtained in this way.\\

A completely different approach, using dynamical systems, was adopted in \cite{Dynamical} to derive sufficient conditions for the existence of a Morse function on a topological manifold.\\

The key idea is to generalise the concept of locally (strictly) convex functions, defined in Section \ref{secdefns}, to topological manifolds, and to use these in the definition of ``Min-type functions'' from \cite{Hyam}. Convexity greatly restricts the type of singularities that can occur. The starting point is the following:

\begin{thm}
\label{thm1}
Let $\mathcal{F}=\{f_{i}\ |\ i\in \mathbb{N}\}$ be a set of convex functions on a Riemannian manifold $M$, and suppose $f:M\rightarrow \mathbb{R}$ is given by $x\mapsto \min\{f_{i}(x)\ |\ f_{i}\in \mathcal{F}\}$. Suppose also that at every $x\in M$, there is a neighbourhood $N_{x}$ of $x$, on which $f$ is the minimum of a finite subset $\mathcal{F}_{x}$ of $\mathcal{F}$, i.e. $f$ is the minimum of a locally finite set of functions. Then $f$ is a topological Morse function.
\end{thm}

Convexity is not defined on topological manifolds. However, Morse functions are defined in terms of local properties. In the case of topological Morse functions, these local properties are preserved by local homeomorphisms. The only properties of convex functions needed in the proof of Theorem \ref{thm1} are satisfied by functions that are locally convex in the image of a chart. This leads directly to the following stronger statement:

\begin{thm}
\label{thm2}
Let $\mathcal{F}=\{f_{i}\ |\ i\in \mathbb{N}\}$ be a set of functions on an $n$-dimensional topological manifold $M$. Suppose that for every $x\in M$, $f:M\rightarrow \mathbf{R}$ is given by $x\mapsto \min\{f_{i}(x)\ |\ f_{i}\in \mathcal{F}\}$. Suppose also that for every $x\in M$, there is a neighbourhood $N_{x}$ of $x$, on which $f$ can be evaluated as the minimum of a finite subset $\mathcal{F}_{x}$ of $\mathcal{F}$, where $N_{x}$ is contained in the domain of a chart in which the elements of $\mathcal{F}_{x}$ map to convex functions on an open set of $\mathbb{R}^{n}$. Then $f$ is a topological Morse function.
\end{thm}

Whenever $f$ is a Morse function, so is $-f$. Theorems \ref{thm1} and \ref{thm2} both hold when the word ``convex'' is replaced by the word ``concave'' and ``minimum'' is replaced by ``maximum''.\\

One reason Theorem \ref{thm2} is much stronger than Theorem \ref{thm1} is that a convex function can have at most one critical point, so convex functions on manifolds with interesting topology do not exist. Theorem 1 is implicit in \cite{MorseSmale} and \cite{SchmutzMorse}, in which an equivariant topological Morse function on a contractible covering space of an orbifold is obtained. \\

Continuous families of topological Morse functions are obtained by observing that a convex function or function that is locally convex in the image of a chart retains its convexity properties when multiplied by a positive constant.\\

Starting perhaps with \cite{G1} and \cite{GS}, one of the major motivations for defining and working with topological Morse functions has been in the study of distance functions on Riemannian manifolds and applications to sphere packings, for example  \cite{Spheres}, \cite{Cheeger}, \cite{Gr} and \cite{Katz}. In this context, a piecewise smooth function is obtained as the distance from some set. Generalisations of the Morse-Sard theorem of such functions are discussed in \cite{IT} and \cite{LR}. Functions of this form are all variants on the idea of Min-type functions studied in \cite{Hyam}. These ideas were generalised to Alexandrov spaces, for example \cite{BGP} and \cite{P2}. Another context in which topological Morse functions described as Min-type functions arise in the literature is the systole function used to study the topology of moduli space, for example \cite{SchmutzMorse} and \cite{MorseSmale}. Some other representative applications of topological Morse functions can be found in \cite{Cech}, \cite{EK}, \cite{Kuiper} and \cite{Saeki}.

\section{Definitions and Examples}
\label{secdefns}

Standard references for topological Morse functions are \cite{Cantwell}, \cite{Siebenmann} and \cite{Morse}.\\

Let $M$ be an $n$-dimensional topological manifold. A \textit{topological Morse function} is a continuous function $f:M\rightarrow \mathbb{R}_{+}$ if the points of $M$ are either regular or critical with respect to $f$. A point $x \in M$ is defined to be regular if there is an open neighbourhood $U$ containing $x$ such that $U$ admits a homeomorphic parametrisation by $n$ parameters, one of which is $f$. If a point $p$ is a critical point there exists a $j\in \mathbb{Z}$, $0\leq j\leq n$, called the \textit{index} of the critical point, and a homeomorphic parametrisation of $U$ by parameters $\{x_{1}, \ldots, x_{n}\}$ such that for every $x$ in $U$, 
\begin{equation*}
f(x)\,-\,f(p) \;=\ \sum_{i=1}^{i=n-j}x_{i}^{2}\,-\,\sum_{i=n-j+1}^{i=n}x_{i}^{2}
\end{equation*}
 
In this paper, a \textit{convex function} $f_{i}$ on a Riemannian manifold is a smooth function for which the sublevel sets of height $t\in \mathbb{R}$, namely $\{x\in \mathbb{M}\ |\ f_{i}\leq t\}$, are strictly convex for all $t\in \mathbb{R}$. Elsewhere in the literature, for example \cite{Convex}, this is also called a smooth, strictly convex function. Equivalently, a smooth function will be called convex if the Hessian is positive definite at every point. The level sets of a convex function locally look like $n-1$ dimensional spheres in $\mathbb{R}^{n}$. A smooth function $f$ on a Riemannain manifold is concave if $-f$ is convex.
 
\begin{ex}[A topological Morse function on $\mathbb{R}^{2}$]
Define $f_{1}$ to be the Euclidean distance in $\mathbb{R}^{2}$ from the point (-1,0) and $f_{2}$ to be the distance in $\mathbb{R}^{2}$ from the point (1,0). The function $f:\mathbb{R}^{2}\rightarrow \mathbb{R}$ given by $\min\{f_{1}, f_{2}\}$ is a topological Morse function. It has two critical points of index 0; these are (-1,0) and (1,0), and a critical point of index 1 at (0,0).
\end{ex}

In \cite{Hyam}, a function $f$ on a manifold $M$ (smooth, Riemannian or topological) was called a Min-type function if there exist finitely many functions $\{f_{1}, \ldots, f_{k}\}$ on $M$ such that at every point $x$ of $M$, $f(x)$ is given by $\min\{f_{1}(x), \ldots, f_{k}(x)\}$. In this paper a slight generalisation will be made, allowing the set of functions $\mathcal{F}=\{f_{i}\}$ to be countably infinite, as long as every point has a neighbourhood on which the minimum only needs to be evaluated over finitely many elements of the set. For example, suppose $X$ is the integer lattice in $\mathbb{R}^{2}$. At a point $x\in \mathbb{R}^{2}$, there are only finitely many distance functions from points in $X$ that realise the minimum at $x$. One only needs to consider finitely many of these when calculating the minimum.

\begin{ex}[A nonexample of a topological Morse function on $\mathbb{R}^{3}$]
\label{localfinitenessex}
Let $f$ be the Euclidean distance from the $x,y$ and $z$ coordinate axes in $\mathbb{R}^{3}$. Then $f$ is not a topological Morse function.
\end{ex}

Example \ref{localfinitenessex} demonstrates that the local finiteness condition in Theorems \ref{thm1} and \ref{thm2} is needed. Example \ref{convexityex} demonstrates the need for the convexity assumption in Theorem \ref{thm1}.

\begin{ex}[A nonexample on $S^{2}$]
\label{convexityex}
Let $p_{1}$ and $p_{2}$ be the two poles of the manifold $S^{2}$ with the standard metric, and $d_{1}:S^{2}\rightarrow \mathbb{R}$ and $d_{2}:S^{2}\rightarrow \mathbb{R}$ be the distance from the points $p_{1}$ and $p_{2}$ respectively. The function $f:S^{2}\rightarrow \mathbb{R}$ given by $\min\{d_{1}, d_{2}\}$ is not a topological Morse function.
\end{ex}

The sublevel set $f_{\leq t}$ at height $t$ of a piecewise-smooth function $f:M\rightarrow \mathbb{R}$ on a smooth Riemannian manifold $M$ is the set $\{x\in M\ |\ f(x)\leq t\}$. The tangent cone at $x$ to $f_{\leq t}$ is the set of vectors in $T_{x}M$ given by tangent vectors at $x$ to smooth curves $\gamma:(-\epsilon,\epsilon)\rightarrow M$ with $\gamma(s)\in f_{\leq t}$ for $s\geq 0$ and $\gamma(0)=x$. For $x$ on the boundary of a locally convex set, the tangent cone at $x$ to the locally convex set is defined analogously. The set of increase of a Min-type function $f$ is defined to be the closure of the subset of $T_{x}M$ consisting of the directions in which $f$ is increasing.\\

Unlike for smooth Morse functions, which have gradients, in the case of the piecewise-smooth Morse functions from Theorem \ref{thm1}, there can exist regular points at which the set of increase has empty interior, i.e. there are no directions in which the function is increasing to first order.

\section{Proofs}
\label{secproofs}
The theorems will now be proven.

\begin{thm}[Theorem \ref{thm1} from the Introduction]
Let $\mathcal{F}=\{f_{i}\ |\ i\in \mathbb{N}\}$ be a set of convex functions on a Riemannian manifold $M$, and suppose $f:M\rightarrow \mathbb{R}$ is given by $x\mapsto \min\{f_{i}(x)\ |\ f_{i}\in \mathcal{F}\}$. Suppose also that at every $x\in M$, there is a neighbourhood $N_{x}$ of $x$, on which $f$ is the minimum of a finite subset $\mathcal{F}_{x}$ of $\mathcal{F}$, i.e. $f$ is the minimum of a locally finite set of functions. Then $f$ is a topological Morse function.
\end{thm}

\begin{proof}
Continuity of $f$ is shown in \cite{Hyam}.\\

The proof consists of showing that all points are either critical or regular according to the definition. Recall that this involves showing that the level sets on a neighbourhood of a critical point of index $r$ are homeomorphic to the level sets on a neighbourhood of a critical point of index $r$ of a smooth Morse function, and the level sets on a neighbourhood of a regular point are homeomorphic to hyperplanes.\\

Suppose $\{f_{1}, \ldots, f_{k}\}\subset \mathcal{F}$ where $f(x)=f_{1}(x)=f_{2}(x)=\ldots=f_{k}(x)$. By local finiteness, there is a neighbourhood of $x$ on which $f$ is obtained as $\min\{f_{1}, \ldots, f_{k}\}$.\\

Local convexity of the functions in the set $\{f_{1}, \ldots, f_{k}\}$ is used here to ensure that, by Theorem 6.2 of \cite{Convex}, the level sets passing through $x$ of each of the functions in the set $\{f_{1}, \ldots, f_{k}\}$ are curved inwards and intersect like spheres. \\

\begin{figure}
\centering
\includegraphics[width=0.4\textwidth]{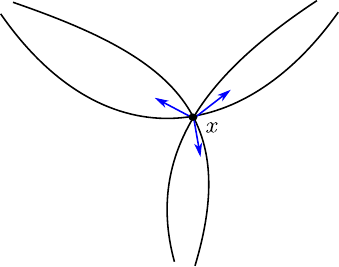}
\caption{The intersection of 3 sublevel sets of convex functions $f_{1}$, $f_{2}$ and $f_{3}$ in $\mathbb{R}^{2}$. The arrows depict the gradients of the functions.}
\label{localmax}
\end{figure}

\begin{figure}
\centering
\includegraphics[width=0.7\textwidth]{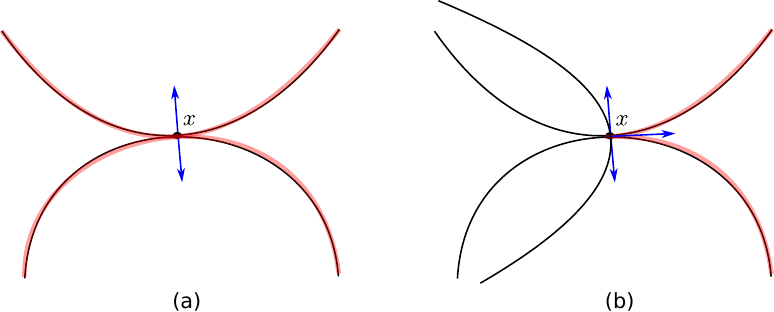}
\caption{Part (a) shows the intersection of sublevel sets of convex functions $f_{1}$ and $f_{2}$ for a critical point $x$ of index 1. Part (b) shows the intersection of sublevel sets of convex functions $f_{1}$, $f_{2}$ and $f_{3}$ in the case of a topologically regular point. The level sets of $f$ through $x$ are shown in red. }
\label{indexr}
\end{figure}

A necessary condition for $x$ to be critical is that there is no direction in which all the functions in $\{f_{1}, \ldots, f_{k}\}$ are increasing to first order at $x$. When $k=1$ and $x$ is the minimum of $f_{1}$, $x$ is a critical point of index 0. The point $x$ is a local maximum if every vector in $T_{x}M$ gives a direction in which at least one of the functions in $\{f_{1}, \ldots, f_{k}\}$ is decreasing. See Figure \ref{localmax}. It is a critical point of index $r$ if $\{\nabla f_{1}(x), \ldots, \nabla f_{k}(x)\}$ spans a subspace $V$ of $T_{x}M$ of dimension $r$, and every vector in $V$ gives a direction in which the length of at least one function in the set $\{f_{1}, \ldots, f_{k}\}$ is decreasing. An example is shown in Figure \ref{indexr} (a). Note that the second condition is essential, because otherwise the point is a topologically regular point, as illustrated in Figure \ref{indexr} (b).\\

It will now be shown that all remaining points are regular points. The key idea here is that the intersection $I$ of the sublevel sets of the functions in $$\{f_{1}, \ldots, f_{k}\}$$ at height $t=f(x)$ is convex on a neighbourhood of $x$. It follows that the tangent cone $C$ to $I$ at $x$ is a cone. This is shown for example at the beginning of Section 4 of \cite{Convex}. Reversing the signs of the vectors in $C$, one obtains the set of increase of $f$ at $x$, which is then also a cone. Note that when this cone has empty interior, the increase of $f$ is to second or higher order as in Figure \ref{indexr} (b). The level set of $f$ through $x$ is therefore locally the boundary of a cone, and hence homeomorphic to a hyperplane as required.
\end{proof}

\begin{thm}[Theorem \ref{thm2} from the Introduction]
Let $\mathcal{F}=\{f_{i}\ |\ i\in \mathbb{N}\}$ be a set of functions on an $n$-dimensional topological manifold $M$. Suppose that for every $x\in M$, $f:M\rightarrow \mathbb{R}$ is given by $x\mapsto \min\{f_{i}(x)\ |\ f_{i}\in \mathcal{F}\}$. Suppose also that for every $x\in M$, there is a neighbourhood $N_{x}$ of $x$, on which $f$ can be evaluated as the minimum of a finite subset $\mathcal{F}_{x}$ of $\mathcal{F}$, where $N_{x}$ is contained in the domain of a chart in which the elements of $\mathcal{F}_{x}$ map to convex functions on an open set of $\mathbb{R}^{n}$. Then $f$ is a topological Morse function.
\end{thm}
\begin{proof}
The definition of topological Morse function is such that it suffices to define the parameterisation in the image of a chart. Moreover, the proof of Theorem \ref{thm1} only required that the functions in $\mathcal{F}(x)$ map to convex functions under the chart in a neighbourhood of $x$. The theorem therefore follows from the same argument as Theorem \ref{thm1}.
\end{proof}

\textbf{Constructing continuous families of topological Morse functions.}
Suppose $f$ is a topological Morse function from Theorem \ref{thm1}. Let $\mathcal{C}$ be an element of $\mathbb{R}^{\mathcal{F}}$ with positive entries close to 1. Every entry of $\mathcal{C}$ corresponds to a function in $\mathcal{F}$. Since a convex function remains convex when multiplied by a positive constant, a new set of convex functions $\mathcal{F}(\mathcal{C})$ is obtained by multiplying every element of $\mathcal{F}$ with the corresponding entry of $\mathcal{C}$.\\

When the entries of $\mathcal{C}$ are sufficiently close to 1, local finiteness from Theorem \ref{thm1} still holds, and $f(\mathcal{C}):=\min\{f(\mathcal{C})_{i}\ |\ f(\mathcal{C})_{i}\in \mathcal{F}(\mathcal{C})\}$ is a topological Morse function. When $f$ is a topological Morse function from Theorem \ref{thm2}, the elements of $\mathcal{C}$ also need to be sufficiently close to 1 to ensure that for every $x\in M$, there remains a chart around $x$ in which the functions $\mathcal{F}_{x}(C) \subset \mathcal{F}(C)$ are locally convex.\\

In conclusion, while topological Morse functions do not appear to be generic, the examples constructed in this paper are also not rigid.

\bibliography{TMFbib}
\bibliographystyle{plain}

\end{document}